\documentclass[11pt]{article}

\usepackage{amsmath}
\usepackage{amsthm}
\usepackage{amssymb}
\usepackage{latexsym}
\usepackage{pstricks}
\usepackage{graphicx, pst-plot, pst-node, pst-text, pst-tree}
\usepackage{titlefoot}
\usepackage{titlesec}
\usepackage[small,it]{caption}
\usepackage{color}
\definecolor{lightgray}{rgb}{0.8, 0.8, 0.8}
\definecolor{darkgray}{rgb}{0.7, 0.7, 0.7}
\definecolor{darkblue}{rgb}{0, 0, .4}

\usepackage[bookmarks,breaklinks]{hyperref}
\hypersetup{
        colorlinks=true,
        linkcolor=darkblue,
        anchorcolor=darkblue,
        citecolor=darkblue,
        pagecolor=darkblue,
        urlcolor=darkblue,
        pdfpagemode=UseThumbs,
        pdftitle={On partial well-order for monotone grid classes of permutations},
        pdfsubject={Combinatorics},
        pdfauthor={Vincent Vatter and Steve Waton},
}

\newtheorem{theorem}{Theorem}[section]
\newtheorem{proposition}[theorem]{Proposition}

\newtheorem{corollary}[theorem]{Corollary}

\newtheorem{question}[theorem]{Question}

\newcommand{\C}{\mathcal{C}}

\newcommand{\M}{\mathcal{M}}

\newcommand{\Grid}{\operatorname{Grid}}
\newcommand{\zpm}{0/\mathord{\pm} 1}

\newcounter{todocounter}

\makeatletter
\newfont{\footsc}{cmcsc10 at 8truept}
\newfont{\footbf}{cmbx10 at 8truept}
\newfont{\footrm}{cmr10 at 10truept}
\renewenvironment{abstract}%
                {
                  \begin{list}{}%
                     {\setlength{\rightmargin}{1in}%
                      \setlength{\leftmargin}{1in}}%
                   \item[]\ignorespaces\begin{small}}%
                 {\end{small}\unskip\end{list}}

\datefoot{\today}
\amssubj{06A07 (primary) and 05A05 (secondary)}

\newpagestyle{main}[\small]{
        \headrule
        \sethead[\usepage][][]
        {\sc On Partial Well-Order For Monotone Grid Classes of Permutations}{}{\usepage}}

\title{\sc{On Partial Well-Order For Monotone Grid Classes of Permutations}}
\author{%
Vincent Vatter\footnote{This research was conducted while V. Vatter was a member of the School of Mathematics and Statistics at the University of St Andrews, supported by EPSRC grant GR/S53503/01.}\\[-0.25ex]
\small Department of Mathematics\\[-0.5ex]
\small Dartmouth College\\[-0.5ex]
\small Hanover, New Hampshire USA\\[1.5ex]
Steve Waton\\[-0.25ex]
\small Department of Mathematics and Statistics\\[-0.5ex]
\small University of St Andrews\\[-0.5ex]
\small St Andrews, Fife, Scotland\\[-10pt]
}

\titleformat{\section}
        {\large\sc}
        {\thesection.}{1em}{}   

\date{}

\begin{document}
\maketitle

\pagestyle{main}

\begin{abstract}
A monotone grid class is a permutation class (i.e., a downset of permutations under the containment order) defined by local monotonicity conditions.  We give a simplified proof of a result of Murphy and Vatter that monotone grid classes of forests are partially well-ordered.  
\end{abstract}

\section{Introduction}\label{grid-pwo-intro}

The permutation $\pi$ of $[n]=\{1,2,\dots,n\}$ {\it contains\/} the permutation $\sigma$ of $[k]=\{1,2,\dots,k\}$ (written $\sigma\le\pi$) if $\pi$ has a subsequence of length $k$ order isomorphic to $\sigma$.  For example, $\pi=391867452$ (written in list, or one-line notation) contains $\sigma=51342$, as can be seen by considering the subsequence $91672$ ($=\pi(2),\pi(3),\pi(5),\pi(6),\pi(9)$).  A {\it permutation class\/} is a downset of permutations under this order; thus if $\C$ is a permutation class, $\pi\in\C$, and $\sigma\le\pi$, then $\sigma\in\C$.  Our interest here is in particular permutation classes called monotone grid classes.

Given a permutation $\pi$ of length $n$ and sets $X,Y\subseteq[n]$, we write $\pi(X\times Y)$ for the permutation that is order isomorphic to the subsequence of $\pi$ with indices from $X$ and values in $Y$.  For example, to compute $136854792([5,9]\times[1,5])$ we consider the subsequence of entries in indices $5$ through $9$, $54792$, which have values between $1$ and $5$; in this case the subsequence is $542$, so $136854792([5,9]\times[1,5])=321$.

\def\symbolfootnote[#1]#2{\begingroup%
\def\thefootnote{\fnsymbol{footnote}}\footnote[#1]{#2}\endgroup}
Suppose that $\M$ is a $t\times u$ matrix%
\symbolfootnote[2]{In order for the cells of a matrix to line up with the corresponding sections in a gridded permutation, we index matrices beginning from the lower left-hand corner and we reverse the rows and columns, so $\M_{3,2}$ denotes for us the entry of $\M$ in the $3$rd column from the left and $2$nd row from the bottom.  Therefore, a $t\times u$ matrix, for our purposes, has $t$ columns and $u$ rows.}
whose entries lie in $\{0,\pm 1\}$.  An {\it $\M$-gridding\/} of the permutation $\pi$ of length $n$ is a pair of sequences $1=c_1\le\cdots\le c_{t+1}=n+1$ (the column divisions) and $1=r_1\le\cdots\le r_{u+1}=n+1$ (the row divisions) such that $\pi([c_k,c_{k+1})\times[r_\ell,r_{\ell+1}))$ is increasing if $\M_{k,\ell}=1$, decreasing if $\M_{k,\ell}=-1$, and empty if $\M_{k,\ell}=0$.  We refer to a permutation together with an $\M$-gridding as an {\it $\M$-gridded permutation\/}.  Figure~\ref{fig-graph-grid-136854792} shows an example.

The {\it grid class of $\M$\/}, written $\Grid(\M)$, consists of all permutations which possess an $\M$-gridding.  Here we have simply defined monotone grid classes; Brignall~\cite{brignall:grid-classes-an:} and Vatter~\cite{vatter:small-permutati:} study more general grid classes.

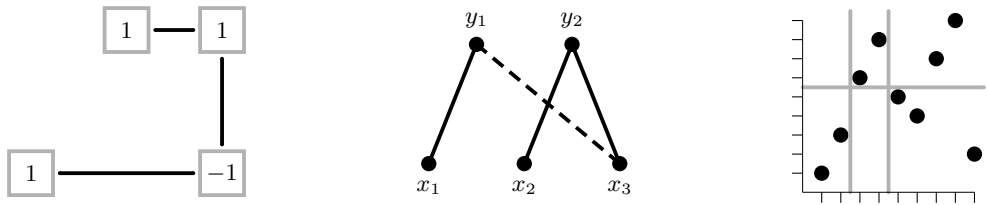
\begin{figure}[t]
\begin{footnotesize}
\begin{center}
\begin{tabular}{ccccc}
\psset{xunit=0.025in, yunit=0.025in}
\psset{linewidth=0.005in}
\begin{pspicture}(0,-3)(50,40)
\psframe[linecolor=darkgray,linewidth=0.02in](0,0)(10,10)
\rput[c](5,5){$1$}
\psframe[linecolor=darkgray,linewidth=0.02in](40,0)(50,10)
\rput[c](45,5){$-1$}
\psframe[linecolor=darkgray,linewidth=0.02in](40,30)(50,40)
\rput[c](45,35){$1$}
\psframe[linecolor=darkgray,linewidth=0.02in](20,30)(30,40)
\rput[c](25,35){$1$}
\psline[linecolor=black,linewidth=0.02in]{c-c}(11,5)(39,5)
\psline[linecolor=black,linewidth=0.02in]{c-c}(45,11)(45,29)
\psline[linecolor=black,linewidth=0.02in]{c-c}(31,35)(39,35)
\end{pspicture}
&\rule{0.5in}{0pt}&
\psset{xunit=0.025in, yunit=0.025in}
\psset{linewidth=0.005in}
\begin{pspicture}(0,0)(50,40)
\pscircle*(5,10){0.04in}
\pscircle*(25,10){0.04in}
\pscircle*(45,10){0.04in}
\pscircle*(15,35){0.04in}
\pscircle*(35,35){0.04in}
\rput[c](5,5){$x_1$}
\rput[c](25,5){$x_2$}
\rput[c](45,5){$x_3$}
\rput[c](15,40){$y_1$}
\rput[c](35,40){$y_2$}
\psline[linecolor=black,linewidth=0.02in](5,10)(15,35) 
\psline[linecolor=black,linewidth=0.02in,linestyle=dashed](45,10)(15,35) 
\psline[linecolor=black,linewidth=0.02in](25,10)(35,35) 
\psline[linecolor=black,linewidth=0.02in](45,10)(35,35) 
\end{pspicture}
&\rule{0.5in}{0pt}&
\psset{xunit=0.01in, yunit=0.01in}
\psset{linewidth=0.005in}
\begin{pspicture}(0,-10)(90,100)
\pscircle*(10,10){0.04in}
\pscircle*(20,30){0.04in}
\pscircle*(30,60){0.04in}
\pscircle*(40,80){0.04in}
\pscircle*(50,50){0.04in}
\pscircle*(60,40){0.04in}
\pscircle*(70,70){0.04in}
\pscircle*(80,90){0.04in}
\pscircle*(90,20){0.04in}
\psline[linecolor=darkgray,linestyle=solid,linewidth=0.02in]{c-c}(0,55)(95,55)
\psline[linecolor=darkgray,linestyle=solid,linewidth=0.02in]{c-c}(25,0)(25,95)
\psline[linecolor=darkgray,linestyle=solid,linewidth=0.02in]{c-c}(45,0)(45,95)
\psaxes[dy=10,Dy=1,dx=10,Dx=1,tickstyle=bottom,showorigin=false,labels=none](0,0)(90,90)
\end{pspicture}
\end{tabular}
\end{center}
\end{footnotesize}
\caption[The cell graph (left) and row-column graph (center) of a matrix of permutation classes.]{The cell (left) and row-column (center) graph of the matrix $\left(
\begin{array}{ccc}
&1&1\\
1&&-1
\end{array}
\right)
$, and, on the right, a gridding of the permutation $136854792$ for this matrix; the column divisions are $c_1,c_2,c_3,c_4=1,3,5,10$ while the row divisions are $r_1,r_2,r_3=1,6,10$.  Here and in what follows, we suppress $0$ entries from our matrices, and in the row-column graph in the center, denote positive edges by solid lines and negative edges by dashed lines.}\label{fig-graph-grid-136854792}
\end{figure}

Monotone grid classes were introduced to study infinite antichains of permutations.  Recall that a poset is said to be {\it partially well-ordered (pwo)\/} if it contains neither an infinite strictly descending chain (which no permutation class can contain) nor an infinite antichain (a set of pairwise incomparable elements).  The pwo properties of a monotone grid class are determined by a certain bipartite graph; the {\it row-column graph\/} of a $t\times u$ matrix $\M$ is the bipartite graph on the vertices $x_1,\dots,x_t,y_1,\dots,y_u$ where $x_i\sim y_j$ if and only if $\M_{i,j}\neq 0$.

\begin{theorem}[Murphy and Vatter~\cite{murphy:profile-classes:}]\label{pwo-forest}
For a $\zpm$ matrix $\M$, $\Grid(\M)$ is pwo if and only if the row-column graph of $\M$ is a forest.
\end{theorem}

In Section~\ref{grid-pwo-pwo} we give a streamlined proof of half of Theorem~\ref{pwo-forest}, showing that $\Grid(\M)$ is pwo whenever the row-column graph of $\M$ is a forest.

It is often more intuitive to consider a different graph, the {\it cell graph\/} of $\M$.  This is the graph on the vertices $\{(i,j) : \M_{i,j}\neq\emptyset\}$ in which $(i,j)\sim(k,\ell)$ if $(i,j)$ and $(k,\ell)$ share a row or a column and there are no nonzero cells between them in this row or column.  We label the vertex $(k,\ell)$ in this graph by the entry $\M_{k,\ell}$.  Figure~\ref{fig-graph-grid-136854792} shows the cell graph of a matrix.  The distinction between cell and row-column graphs does not affect Theorem~\ref{pwo-forest}:

\begin{proposition}
The row-column graph of the $\zpm$ matrix $\M$ is a forest if and only if the cell graph of $\M$ is a forest.
\end{proposition}
\begin{proof}
Consider a cycle $x_{k_1}\sim y_{\ell_1}\sim \cdots\sim x_{k_m}\sim y_{\ell_m}\sim x_{k_1}$ in the row-column graph of $\M$ (as this graph is bipartite, such a cycle must be of even length and alternate between $x$ and $y$ vertices).  By definition, this means that $\M_{k_1,\ell_1},\dots,\M_{k_m,\ell_m},\M_{k_1,\ell_m}\neq 0$.  This does not necessarily mean that $(k_1,\ell_1)\sim\cdots\sim(k_m,\ell_m)\sim(k_1,\ell_m)\sim(k_1,\ell_1)$ in the cell graph, because there may be nonempty cells in between these cells, but by including these cells we find a cycle in the cell graph of $\M$.  In the other direction, such cells may similarly be removed.
\end{proof}

\section{Partial Multiplication Matrices}\label{grid-pwo-part-mult}

We say that the $t\times u$ $\zpm$ matrix $\M$ is a {\it partial multiplication matrix\/} if we can choose {\it row signs\/} $r_1,\dots,r_u\in\{\pm 1\}$ and {\it column signs\/} $c_1,\dots,c_t\in\{\pm1\}$ such that for all cells $(k,\ell)\in[t]\times [u]$, $\M_{k,\ell}$ is equal to either $c_kr_\ell$ or $0$.  Because there is a unique path between every two vertices of a tree, $\M$ is a partial multiplication matrix whenever $G_\M$ is a forest.

In order to characterize all the partial multiplication matrices, we need to make a definition.  A cycle $x_{k_1}\sim y_{\ell_1}\sim \cdots\sim x_{k_m}\sim y_{\ell_m}\sim x_{k_1}$ in the row-column graph of the $\zpm$ matrix $\M$ is {\it positive\/} (resp., {\it negative\/}) if $\M_{k_1,\ell_1}\M_{k_2,\ell_1}\M_{k_2,\ell_2}\cdots\M_{k_m,\ell_m}\M_{k_1,\ell_m}=1$ (resp., $-1$).

\begin{proposition}\label{part-mult-characterization}
The $\zpm$ matrix $\M$ can be expressed as a partial multiplication matrix if and only if its row-column graph contains no negative cycles.
\end{proposition}
\begin{proof}
Suppose first that $\M$ is a partial multiplication matrix with respect to the row signs $r_1,\dots,r_u$ and column signs $c_1,\dots,c_t$ and consider a cycle $x_{k_1}\sim y_{\ell_1}\sim \cdots\sim x_{k_m}\sim y_{\ell_m}\sim x_{k_1}$ in its row-column graph.  The parity of this cycle is
\begin{eqnarray*}
\M_{k_1,\ell_1}\M_{k_2,\ell_1}\M_{k_2,\ell_2}\cdots\M_{k_m,\ell_m}\M_{k_1,\ell_m}
&=&
(c_{k_1}r_{\ell_1})(c_{k_2}r_{\ell_1})(c_{k_2}r_{\ell_2})\cdots(c_{k_m}r_{k_m})(c_{k_1}r_{\ell_m}),\\
&=&
c_{k_1}^2r_{\ell_1}^2c_{k_2}^2r_{\ell_2}^2\cdots c_{k_m}^2r_{\ell_m}^2,\\
&=&
1,
\end{eqnarray*}
as desired.

For the other direction, suppose that $\M$ is a $t\times u$ $\zpm$ matrix whose row-column graph contains only cycles of positive sign.  This ensures that any two disjoint paths between two vertices have the same sign, from which it follows that any two paths between two vertices have the same sign.  It suffices to prove the claim for every connected component of the row-column graph of $\M$, so we assume that the row-column graph of $\M$ is connected.

For each $k\in[t]$, let $c_k$ denote the sign of any (and hence every) path from $x_1$ to $x_k$, and for each $\ell\in[u]$, let $r_\ell$ denote the sign of any path from $x_1$ to $y_\ell$.  Since we have observed that any two paths from $x_1$ to another vertex have the same sign, $c_k$ and $r_\ell$ are well-defined.  It remains only to show that $\M_{k,\ell}=c_kr_\ell$ if $\M_{k,\ell}\neq 0$.  Take any path from $x_1$ to $x_k$ (which has sign $c_k$).  If $\M_{k,\ell}\neq 0$, then $x_k\sim y_\ell$ in the row-column graph of $\M$, so we may extend this path to $x_\ell$.  The sign of the resulting path is $r_\ell$ by definition, so $r_\ell=c_k\M_{k,\ell}$, which shows that $c_kr_\ell=c_k^2\M_{k,\ell}=\M_{k,\ell}$.
\end{proof}

\section{Order-Preserving Maps and Partial Well-Order}\label{grid-pwo-pwo}

We phrase our proof that monotone grid classes of forests are pwo in the language of order preserving maps; recall that the map $\varphi:P\rightarrow Q$ is {\it order-preserving\/} if $x\le y$ in $P$ implies that $\varphi(x)\le\varphi(y)$.  By the following observation, partial well-order is preserved by onto order-preserving maps.

\begin{proposition}\label{prop-hom-pwo}
Let $P$ and $Q$ be posets.  If $\varphi:P\rightarrow Q$ is order-preserving and onto and $P$ is pwo, then $Q$ is pwo as well.
\end{proposition}
\begin{proof}
Consider any antichain $A\subseteq Q$.  As $\varphi$ is onto, we may choose a pre-image for each $a\in A$.  Since $A$ is an antichain, the chosen pre-images must also form an antichain.  Therefore $A$ must be finite because $P$ is pwo.  It follows similarly that $Q$ has no infinite strictly descending chain, so $Q$ is pwo.
\end{proof}

In this note, we apply Proposition~\ref{prop-hom-pwo} to maps from subword orders to permutation classes.  Recall that given an alphabet $\Sigma$, the {\it subword order\/} (or, {\it scattered subword order\/} or {\it subsequence order\/}) on the set $\Sigma^*$ of all words (i.e., sequences) over $\Sigma$ is defined by $v\le w$ if $v$ occurs as a not-necessarily-contiguous subword (i.e., subsequence) of $w$.  The following result is due originally to Higman~\cite{higman:ordering-by-div:} but has been rediscovered many times since.

\newtheorem*{higmans-theorem}{Higman's Theorem}
\begin{higmans-theorem}
Let $\Sigma$ be a finite alphabet.  Under the subword order, $\Sigma^*$ is pwo.
\end{higmans-theorem}

Now suppose that $\M$ is a partial multiplication matrix with column signs $c_1,\dots,c_t$ and row signs $r_1,\dots,r_u$, and let $\Sigma=\{(k,\ell) : \M_{k,\ell}\neq0\}$.  We aim to define a map $\varphi:\Sigma^*\rightarrow\Grid(\M)$, by describing how to create an $\M$-gridded permutation, $\varphi(w)$, for each word $w\in\Sigma^*$, in which the letter $(k,\ell)$ corresponds to an entry in the $(k,\ell)$ cell.  In order to build this permutation, we read $w$ from left to right, inserting the entries in column $k$ from left-to-right if $c_k=1$ and from right-to-left if $c_k=-1$, while inserting the entries in row $\ell$ from bottom-to-top if $r_\ell=1$ and from top-to-bottom if $r_\ell=-1$.  The following chart summarizes the insertion procedure.
\begin{center}
\begin{tabular}{r|r|l}
$c_k$&$r_\ell$&order of insertion\\\hline
$1$&$1$&left-to-right and bottom-to-top\\
$1$&$-1$&left-to-right and top-to-bottom\\
$-1$&$1$&right-to-left and bottom-to-top\\
$-1$&$-1$&right-to-left and top-to-bottom
\end{tabular}
\end{center}
Note that the entries inserted into the cell $(k,\ell)$ are increasing if $\M_{k,\ell}=1$ and decreasing if $\M_{k,\ell}=-1$.

\begin{proposition}\label{prop-mono-grid-hom}
The map $\varphi:\Sigma^*\rightarrow\Grid(M)$ is order-preserving.
\end{proposition}
\begin{proof}
It is clear that by deleting a letter of $w\in\Sigma^*$ we obtain a (gridded) permutation strictly contained in $\varphi(w)$, as desired.
\end{proof}

While we have defined $\varphi$ for all partial multiplication matrices, it cannot be onto when $\M$ is not a forest, because then $\Grid(\M)$ contains an infinite antichain.  However, as we show next, it is onto otherwise.

\begin{proposition}\label{prop-mono-grid-onto}
If the row-column graph of $\M$ is a forest then the map $\varphi:\Sigma^*\rightarrow\Grid(M)$ is onto.
\end{proposition}
\begin{proof}
Suppose that the row-column graph of the $t\times u$ $\zpm$ matrix $\M$ is a forest.  Since this implies that $\M$ is a partial multiplication matrix, we may select row and column signs $c_1,\dots,c_t$ and $r_1,\dots,r_u$ for $\M$.  Now take $\pi\in\Grid(\M)$ of length $n$.

We would like to describe a word $w\in\Sigma^*$ such that $\varphi(w)=\pi$.  We begin by fixing an $\M$-gridding of $\pi$ and, for each $k\in[t]$, we define the linear {\it column order\/} $\le^{\rm col}_{k}$ on the entries of $\pi$ in column $k$ by
\begin{eqnarray*}
\pi(i)\le^{\rm col}_{k}\pi(j)&\iff&
\left\{\begin{array}{l}
\mbox{$i\le j$ if $c_k=1$, or}\\
\mbox{$i\ge j$ if $c_k=-1$}
\end{array}\right.
\end{eqnarray*}
In other words, $\pi(i)\le^{\rm col}_{k}\pi(j)$ if the letter corresponding to $\pi(i)$ should occur before the letter corresponding to $\pi(j)$ in our construction of $\varphi(w)$.  Similarly, for each $\ell\in[u]$, we define the linear {\it row order\/} $\le^{\rm row}_{\ell}$ on the entries of $\pi$ in row $\ell$ by
\begin{eqnarray*}
\pi(i)\le^{\rm row}_{\ell}\pi(j)&\iff&
\left\{\begin{array}{l}
\mbox{$\pi(i)\le\pi(j)$ if $r_\ell=1$, or}\\
\mbox{$\pi(i)\ge\pi(j)$ if $r_\ell=-1$}
\end{array}\right.
\end{eqnarray*}

Recall that a set of linear orders is {\it consistent\/} if their union forms a poset, which occurs if and only if their union does not contain a cycle.  Clearly the union of row and column orders just described cannot have a cycle of length two: having both $\pi(i)\le^{\rm col}_{k}\pi(j)$ and $\pi(j)\le^{\rm row}_{\ell}\pi(i)$ would contradict the fact that $\M$ is a partial multiplication matrix.  Thus we need only worry about longer cycles.  If $\pi(h)\le^{\rm col}_{k}\pi(i)\le^{\rm row}_{\ell}\pi(j)$ or $\pi(h)\le^{\rm row}_{\ell}\pi(i)\le^{\rm col}_{k}\pi(j)$ then $\M_{k,\ell}\neq 0$.  Therefore, if the union of these linear orders contains a cycle of length greater than two then the row-column graph of $\M$ must contain a cycle, a contradiction.

Because the union of row and column orders is consistent, this union has a linear extension, say $\pi(i_1)\le_*\pi(i_2)\le_*\cdots\le_*\pi(i_n)$.  The word $w\in\Sigma^*$ in which the $j$th letter is equal to the cell of $\pi(i_j)$ satisfies $\varphi(w)=\pi$, verifying that $\varphi$ is onto.
\end{proof}

\begin{corollary}
If the row-column graph of $\M$ is a forest, then $\Grid(\M)$ is pwo.
\end{corollary}
\begin{proof}
This follows immediately from Propositions~\ref{prop-hom-pwo}--\ref{prop-mono-grid-onto} and Higman's Theorem.
\end{proof}

\section{An Example}

\begin{figure}
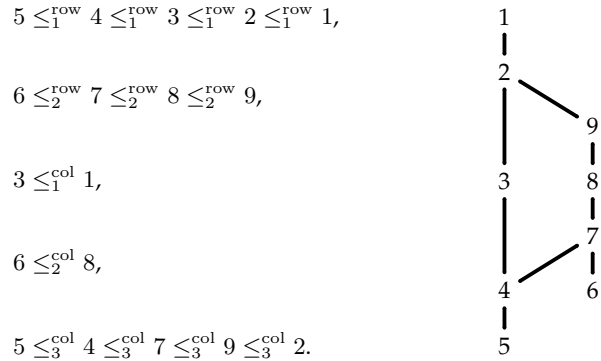

\begin{footnotesize}
\begin{center}
\begin{tabular}{ccc}
\psset{rowsep=0.304in}
\psset{mcol=l}
\begin{psmatrix}
$5\le^{\rm row}_1 4\le^{\rm row}_1 3\le^{\rm row}_1 2\le^{\rm row}_1 1$,\\
$6\le^{\rm row}_2 7\le^{\rm row}_2 8\le^{\rm row}_2 9$,\\
$3\le^{\rm col}_1 1$,\\
$6\le^{\rm col}_2 8$,\\
$5\le^{\rm col}_3 4\le^{\rm col}_3 7\le^{\rm col}_3 9\le^{\rm col}_3 2$.
\end{psmatrix}
&\rule{0.5in}{0pt}&
\psset{linewidth=0.02in}
\psset{rowsep=0.2in}
\psset{colsep=0.4in}
\psset{nodesep=0.05in}
\begin{psmatrix}
1\\
2\\
&9\\
3&8\\
&7\\
4&6\\
5
\ncline{c-c}{1,1}{2,1}
\ncline{c-c}{2,1}{3,2}
\ncline{c-c}{2,1}{4,1}
\ncline{c-c}{3,2}{4,2}
\ncline{c-c}{4,2}{5,2}
\ncline{c-c}{4,1}{6,1}
\ncline{c-c}{5,2}{6,1}
\ncline{c-c}{5,2}{6,2}
\ncline{c-c}{6,1}{7,1}
\end{psmatrix}
\end{tabular}
\end{center}
\end{footnotesize}
\caption{On the left, the row and column orders for the gridding of the permutation $136854792$ shown in Figure~\ref{fig-graph-grid-136854792}, on the right, Hasse diagram of the union of these orders.}
\label{fig-orders-136854792}
\end{figure}

In order to illustrate our proof of Proposition~\ref{prop-mono-grid-onto}, we describe the map
$$
\varphi:\{(1,1),(2,2),(3,1),(3,2)\}^*\rightarrow\Grid\left(\begin{array}{ccc}&1&1\\1&&-1\end{array}\right)
$$
and show that the gridding of the permutation $136854792$ shown in Figure~\ref{fig-graph-grid-136854792} lies in the image of this map.  There are two possible choices of row and column signs for this matrix; choosing $r_1=-1$, $r_2=1$, $c_1=-1$, $c_2=c_3=1$ allows us to describe $\varphi$ recursively by the following rules:
\begin{itemize}
\item upon reading a $(1,1)$, insert a new minimum element to the left of $\pi$,
\item upon reading a $(2,2)$, insert a new maximum element before all entries corresponding to $(3,1)$s and $(3,2)$s,
\item upon reading a $(3,1)$, insert a new minimum element to the right of $\pi$,
\item upon reading a $(3,2)$, insert a new maximum element to the right of $\pi$.
\end{itemize}

In order to construct a word which maps to the gridding of $136854792$ shown in Figure~\ref{fig-graph-grid-136854792}, we begin by constructing the row and column orders.  For example, because $r_1=-1$, the first row order ($\le^{\rm row}_1$) records the entries in the first row ($1$, $2$, $3$, $4$, and $5$) from top-to-botton.  All five such row and column orders are shown on the left of Figure~\ref{fig-orders-136854792}.  Because we began with an acyclic matrix, we are guaranteed that the union of these orders is consistent; this union is depicted on the right of Figure~\ref{fig-orders-136854792}.  Finally, in order to construct a word which is mapped to this permutation, we may choose any linear extension of this poset.  By choosing the linear extension
$$
5\le_*4\le_*6\le_*7\le_*3\le_*8\le_*9\le_*2\le_*1
$$
and recording the cells of each of these entries, we obtain the word
$$
(3,1)\ (3,1)\ (2,2)\ (3,2)\ (1,1)\ (2,2)\ (3,2)\ (3,1)\ (1,1).
$$

\section{Conclusion}

The main result of this note is the existence of order-preserving maps from subword posets to monotone grid classes of partial multiplication matrices.  For monotone grid classes of forests, we have shown that these map are onto, therefore giving a new proof that these permutation classes are partially well-ordered.  One of the more intriguing questions arising from this work is how far similar approaches could be extended:

\begin{question}\label{question-subword-like}
For which permutation classes $\C$ does there exist an onto order-preserving map $\varphi:\Sigma^*\rightarrow\C$ for some finite alphabet $\Sigma$?
\end{question}

Let us say that such classes are {\it subword-like\/}.  While all monotone grid classes of forests are subword-like, there are other subword-like classes.  Two such examples are the class of permutations that can be drawn on a circle, studied by Vatter and Waton~\cite{vatter:on-points-drawn:}, and the permutations that can be drawn on an $\textsf{X}$, examined by Elizalde~\cite{elizalde:the-x-class-and:}.

\bibliographystyle{acm}
\bibliography{../refs}

\end{document}